\documentclass[reqno]{amsart}
\usepackage{amssymb, latexsym, amsmath, amsfonts, amscd}
\usepackage{lscape,color}
\usepackage{mathrsfs}
\usepackage{hyperref}
\usepackage{verbatim}
\usepackage{appendix}
\usepackage{mathtools}

\newtheorem{lemma}{Lemma}
\newtheorem{theorem}{Theorem}
\newtheorem{proposition}{Proposition}

{\theoremstyle{definition}\newtheorem{remark}{Remark}}
{\theoremstyle{definition}\newtheorem{assumption}{Assumption}}
\numberwithin{lemma}{section}
\numberwithin{proposition}{section}

\newcommand{\R}{\mathbf{R}}
\newcommand{\Q}{\mathbf{Q}}

\newcommand{\T}{\mathbf{T}}
\newcommand{\Z}{\mathbf{Z}}
\newcommand{\N}{\mathbf{N}}

\newcommand{\supp}{\textnormal{supp }}

\newcommand{\Span}{\textnormal{Span}}

\title[Hopf-Like Bifurcation at a Removable Singularity]{Hopf-Like Bifurcation in a Wave Equation at a Removable Singularity}

\author[N. Kosovali\'c]{Nemanja Kosovali\'c}
\address{Shopify (formerly University of South Alabama)\\ n.kosovalic@gmail.com}

\author[B. Pigott]{Brian Pigott}
\address{Wofford College\\pigottbj@wofford.edu}

\date{\today}
\subjclass{(Primary) (Secondary) }

\begin{document}

\begin{abstract}
It is shown that a one-dimensional damped wave equation with an odd time derivative nonlinearity exhibits small amplitude bifurcating time periodic solutions, when the bifurcation parameter is the linear damping coefficient is positive and accumulates to zero. The upshot is that the singularity of the linearized operator at criticality which stems from the well known small divisor problem for the wave operator, is entirely removed without the need to exclude parameters via Diophantine conditions, nor the use of accelerated convergence schemes. Only the contraction mapping principle is used. 
\end{abstract}

\maketitle

\section{Introduction}

Time periodic solutions of nonlinear wave equations, have
been a large focus of research spanning the last fifty years,
with major interest in non-dissipative systems, or dissipative systems
subject to a periodic forcing. In the last thirty years, the emphasis
of many works addressing the bifurcation of time periodic solutions of nonlinear wave equations 
is the resolution of small divisor problems. For instance see
\cite{bourg,bourgain1995construction} and
\cite{berti2010sobolev,berti2014kam}, to name a few. In this paper, we consider the problem of bifurcating small amplitude time periodic solutions for nonlinear wave equations when the (positive) linear damping coefficient is used as the bifurcation parameter, and is allowed to be {\textit{arbitrarily small}}. In order for nontrivial time periodic solutions to persist, we consider a nonlinear negative damping as in, for instance, the investigation of the Beam equation in \cite{kosovalic2018quasi}, to balance out the positive linear damping. The linearized problem at criticality (when the damping coefficient is zero) naturally contains a singularity owing to the small divisor problem stemming from the wave operator. The upshot of this work is that the singularity can be entirely removed {\textit{without}} the use of techniques such as Nash-Moser iteration, nor the exclusion of parameters via Diophantine conditions. In order to remove the singularity, we exploit three critical features of our problem: (i) An asymptotic expansion for the lower bound of the damping parameter in terms of amplitude, (ii) the fact a certain integral vanishes (see Lemma \ref{large_k_zeros}), and (iii) the regularizing effect of Laplacian damping. This allows us to solve the infinite dimensional part of the problem using merely the contraction mapping principle. Our work is motivated by \cite{bambusi2000lyapunov} which proves Lyapunov Center Theorems for various nonlinear wave equations using the contraction mapping principle but differs in two ways: We have an explicit bifurcation parameter, so we are really proving a Hopf Bifurcation Theorem and not a Lyapunov Center Theorem, and we do not require any Diophantine conditions to restrict parameters.

We consider the following damped nonlinear wave equation with Dirichlet boundary conditions:

\begin{equation}
\label{dnlw}
\left \{ \begin{array}{ll}
\partial_{t}^{2} u - \partial_{x}^{2}u - \alpha \partial_{t} \partial_{x}^{2} u + mu = \Big ( \partial_{t} u \Big )^{2p+1}, & t \in \R, x \in \T = \R/2\pi \Z,\\
u(t,0) = u(t,\pi) = 0.
\end{array}
\right .
\end{equation}

In \eqref{dnlw} $p \geq 1$ is taken to be an integer, $\alpha > 0$ and $m > 0$ is irrational. 

\smallskip

Before stating the main results we make some clarifying remarks concerning our choices of the nonlinearity, damping, and the question of the stability of the obtained periodic solutions:

\begin{enumerate}
    \item \textit{{Nonlinearity}} - Due to the presence of positive linear damping, in order for non-trivial periodic solutions to persist, we need the nonlinearity to restore energy to the system, and the odd time derivative nonlinearity is the simplest autonomous mechanism which achieves this. See the energy argument in \cite{kosovalic2018quasi} for details. Such velocity dependent nonlinearities arise in the phenomena of self-excited vibrations; see the introduction of \cite{kosovalic2018quasi} and the references therein. In fact, our results can be extended to include higher order polynomial perturbations on top of the present nonlinearity, which do not contain any derivatives. On a related note the Lemma \ref{large_k_zeros} which concerns the vanishing of a certain integral, is really just another way of stating that monomial expansions of trigonometric functions have finitely many nonzero Fourier modes, and has nothing special to do with our choice of nonlinearity. 
    \item  \textit{{Damping}} - The choice of Laplacian damping (sometimes called Kelvin-Voigt damping) in this work is used to facilitate removing the singularity. It is {\textit{not}} clear if our results can be obtained for non-Laplacian linear damping using the same techniques. On the other hand, our choice of damping  naturally arises in modelling the vibrations of viscoelastic materials, as an internal friction of the vibrating material. See e.g. the introduction of \cite{ammari2020stabilization} for a discussion of damping mechanisms in wave equations.
    \item \textit{{Stability}} - We do {\textit{not}} obtain any results concerning the stability of the periodic solutions here, but can comment on the direction of bifurcation. Although the solutions we obtain persist for arbitrarily small positive values of the damping coefficient, a change of variable $t\to -t$ allows us to get solutions for negative values of the damping coefficient for the nonlinearity $-1 \cdot  (\partial_{t} u)^{2p+1}$. Provided the principle of exchange of stability holds (see Section I.12 of the monograph \cite{kielhofer2006bifurcation}), then we expect the latter branch of periodic solutions to be stable. 
\end{enumerate}

\subsection{Statement of Results}

Our main result is 
\begin{theorem}[Small Damping Bifurcation]  \label{maintheorem}
Let $k_0\ge 2$ be given. Then there is some $M=M(p,m,k_0)>0$ so that for each $\rho \in(0,M)$ there is some $\alpha \gtrsim_{m,p,k_0} \rho^{2p}$ and $\omega \approx \sqrt{1+m}$ so that the nonlinear wave equation \eqref{dnlw} has a $C^{k_0}$ in $(t,x)$ $\frac{2\pi}{\omega}$-time periodic solution of the form $$u(t,x)=\rho \cos(\omega t)\sin(x) +v(\omega t,x)$$ where $(t,x)\mapsto v(t,x)$ is a $2\pi$-time periodic function, and $v(t,x) = O(\rho^{2p+\frac{1}{2}})$ with respect to a suitable Sobolev norm. 
\end{theorem}

\subsection{Organization}
The paper is organized as follows. In Section \ref{SectionNotation} we introduce the basic notation that will be used throughout the paper as well as the Hilbert space $X^{s}$ in which we will work. We discuss the periodic solution ansatz and develop some key properties of the linearized operator in Section \ref{SectionAnsatz}. In Section \ref{SectionLSReduction} we introduce the Lyapunov-Schmidt decomposition, solve the range equation and the bifurcation equations, and prove Theorem \ref{maintheorem}. 

\section{Notation}
\label{SectionNotation}
We begin by introducing some notation that we use throughout the paper:

\begin{itemize}
	\item $A \lesssim B$ means that there is a universal constant $C$ so that $A \leq CB$. In cases where we wish to indicate that the implicit constant $C$ depends on a variable $q$ we will write $A \lesssim_{q} B$; that is, $A \lesssim_{q} B$ means $A \leq C(q) B$.
	\item $A \sim B$ means that $A \lesssim B$ and $B \lesssim A$.
	\item $x \approx y$ means that there is an $\epsilon > 0$ for which $\vert x - y \vert < \epsilon$.
	\item $\displaystyle \mathcal{L}(X,Y)$ is the set of bounded linear maps from $X \to Y$ (where $X$ and $Y$ are Banach spaces).
	\item $\mathscr{B}_{X}(x,r)$ denotes the ball of radius $r > 0$ centered at the point $x$ in the metric space $X$. When the space $X$ is clear from the context, we will write $\mathscr{B}(x,r)$.
 \item We use the symbol $D$ to denote the Fr\'echet derivative for mappings between Banach spaces.
\end{itemize}

\subsection{The Hilbert Space $X^{s}$}

We recall the definition of the space $X^{s}$ that was introduced in \cite{kosovalic2018self}. The space $X^{s}$ is the subset of $H^{s}(\T^{2})$ consisting of all real-valued functions that can be written in the form
\begin{equation*}
u(t,x) = \sum_{\substack{n \in \Z\\ k \geq 1}} c_{n,k} e^{int} \sin(kx).
\end{equation*}
That is,
\begin{equation}
\label{Xsdefn}
X^{s} := \left \{ \sum_{\substack{n \in \Z\\ k \geq 1}} c_{n,k} e^{int} \sin(kx) \  \Big \vert  \ \overline{c_{n,k}} = c_{-n,k}, \ \sum_{\substack{n \in \Z\\ k \geq 1}} (n^{2s} + k^{2s}) \vert c_{n,k} \vert^{2} < \infty \right   \}.
\end{equation}
We recall that if $s \geq 2$, then $X^{s}$ is a closed subspace of $H^{s}(\T^{2})$ whose members are real-valued continuous functions. In fact, the space $X^{s}$ is a real Hilbert space whose inner product is given by
\begin{equation*}
\langle u, v \rangle_{X^{s}} = \int_{\T^{2}} D_{t}^{s} u(t,x) D_{t}^{s} v(t,x) dt dx + \int_{\T^{2}} D_{x}^{s} u(t,x) D_{x}^{s} v(t,x) dt dx
\end{equation*}
with corresponding norm given by $\| u \|_{X^{s}}^{s} = \langle u, u \rangle_{X^{s}}$. We point out that this norm is equivalent to
\begin{equation*}
\sum_{\substack{n \in \Z\\ k \geq 1}} (n^{2s} + k^{2s}) \vert c_{n,k} \vert^{2}.
\end{equation*}
As $k \geq 1$ in the index of summation in the definition of $X^{s}$ and its norm, we find that this norm is also equivalent to the norm in $H^{s}(\T^{2})$ for members of $X^{s}$.

We point out that if $u \in X^{s}$, then $u(t,x)$ can be written in the form
\begin{equation*}
u(t,x) = \sum_{k \geq 1} c_{k} \sin(kx) + \sum_{\substack{n \geq 1\\ k \geq 1}} \Big ( d_{n,k} \cos(nt) \sin(kx) + e_{n,k} \sin(nt) \sin(kx) \Big ),
\end{equation*}
where the Fourier coefficients $c_{k}, d_{n,k}, e_{n,k} \in \R$.

\subsubsection{Sobolev Embedding and Banach Algebra}
\label{cmpt_emb}

With this notation, we have the following Sobolev embedding theorem: If $k \ge 0$
and $s > k + 1$ then $H^s(\T^2)\subset C^k(\T^2)$ and $ \| u \|_{C^k}\lesssim\| u  \|_{H^s}$ and hence the same holds for $X^s$ in place of $H^s$. Further if $s>1$ then $H^s(\T^2)$ is a Banach algebra via multiplication of functions, which is inherited by space $X^{s}$. That is, for $v,w\in X^s$  
\begin{equation}\label{algebra}
\|vw\|_{X^{s}}\le 2^{2s} \|v\|_{X^{s}}\|w\|_{X^{s}}.
\end{equation}

\section{Motivating Calculations and Periodic Solution Ansatz}
\label{SectionAnsatz}

We seek a solution of \eqref{dnlw} which is periodic in $t$. As in \cite{kosovalic2018self, nosymm2019, kospigott_highd_sym} we begin with the ansatz for the Lyapunov-Schmidt reduction in $X^{s}$. To find a $\frac{2\pi}{\omega}$-periodic solution of \eqref{dnlw} it suffices to find a $2\pi$-periodic solution of the following problem:
\begin{equation}
\label{omegadnlw}
\left \{
\begin{array}{ll}
\omega^{2} \partial_{t}^{2} u^{\omega} - \partial_{x}^{2} u^{\omega}  -\omega \alpha \partial_{t}\partial_{x}^{2} u^{\omega} + mu^{\omega} = \omega^{2p+1} \big ( \partial_{t} u^{\omega} \big )^{2p+1}\\
u^{\omega}(t,0) = u^{\omega}(t,\pi) = 0
\end{array}
\right .
\end{equation}
(We require, of course, that $\omega > 0$.)
In the sequel we will omit the $\omega$ superscript and write $u$ for the solution of \eqref{omegadnlw}.

That having been said in order to prove Theorem \ref{maintheorem}, it suffices to prove the following auxiliary result:

\begin{theorem}[Small Damping Bifurcation]  \label{aux_maintheorem}
Let $k_0\ge 2$ be given. Then there is some $M=M(p,m,k_0)>0$ so that for each $\rho \in(0,M)$ there is some $\alpha \gtrsim_{m,p,k_0} \rho^{2p}$ and $\omega \approx \sqrt{1+m}$ so that \eqref{omegadnlw} has a $C^{k_0}$ in $(t,x)$ $2\pi$-time periodic solution of the form $$u(t,x)=\rho \cos(t)\sin(x) +O(\rho^{2p+\frac{1}{2}}).$$
\end{theorem}

Upon linearizing \eqref{omegadnlw} we  look for solutions that have the form
\begin{equation}
\label{ansatz}
u(t,x) = \sum_{n \in \Z} \sum_{k \geq 1} c_{n,k} e^{int} \sin(kx).
\end{equation}
Because the solution is $\R$-valued the coefficients $c_{n,k}$ satisfy $\overline{c_{n,k}} = c_{-n,k}$.
Inserting the ansatz \eqref{ansatz} into the linearized equation (around the zero steady state $u=0$) yields
\begin{equation*}
-n^{2} \omega^{2} + k^{2} + i \alpha \omega n k^2 + m = 0.
\end{equation*}
In taking real and imaginary parts in this equation we find the following system of equations:
\begin{align}
-n^{2} \omega^{2} + k^{2} + m &= 0 \label{ansatzeq1}\\
\alpha \omega n k^2 &= 0 \label{ansatzeq2}.
\end{align}
Notice that \eqref{ansatzeq2} is automatically satisfied if $\alpha = 0$. Turning to \eqref{ansatzeq1} we see that if we take $n = \pm 1, k = 1$, then $\omega = \sqrt{1 + m}$. It follows that if $m > 0$ is taken to be irrational, then $\omega_{1} := \sqrt{1 + m}$ is irrational. With this in mind we fix $m > 0$ to be irrational .

Consider the linear operator
\begin{equation}
\label{linearization}
L_{\omega,\alpha} u := \omega^{2} \partial_{t}^{2} u - \partial_{x}^{2} u - \omega \alpha \partial_{t}\partial_{x}^{2} u + mu.
\end{equation}
With this notation we find that \eqref{omegadnlw} can be written
\begin{equation*}
L_{\omega, \alpha} u = \omega^{2p+1} \big ( \partial_{t} u \big )^{2p+1}.
\end{equation*}

\begin{lemma} \label{kernellemma}
The kernel of $L_{\omega_{1},0}$ in the space $X^s$ is two-dimensional: 
\begin{equation*}
\ker L_{\omega_{1},0} = \Span \{ e^{it} \sin(x), e^{-it} \sin(x) \}.
\end{equation*}
\end{lemma}

\begin{proof}
	A function $u(t,x)$ of the form \eqref{ansatz} satisfies $L_{\omega_{1},0} u = 0$ if and only if $m - n^{2} \omega_{1}^{2} + k^{2} = 0$. Upon solving for $\omega_{1}^{2}$ and recalling that $\omega_{1} = \sqrt{1 + m}$ we find that
	\begin{equation*}
	\frac{m + k^{2}}{n^{2}} = 1 + m.
	\end{equation*}
	This can be rearranged to read $k^{2} - n^{2} = m(n^{2} - 1)$. Since $m \in (\R \setminus \Q)^{+}$ we must have $n = \pm 1$ and, in that case, $k = 1$.
\end{proof}

We write $X^{s} = \ker L_{\omega_{1},0} \oplus V^{s}$, the space $V^{s}$ being the orthogonal complement of $\ker L_{\omega_{1},0}$ (with respect to the $L^{2}$ inner product). Taking $A \subset \Z \times \N$ to be the set
\begin{equation*}
A := \{ (n,k) \in \Z \times \N \ \vert \ (n,k) \neq (\pm 1, 1) \},
\end{equation*}
we see that $u \in V^{s}$ if
\begin{equation*}
u(t,x) = \sum_{(n,k) \in A} c_{n,k} e^{int} \sin(kx).
\end{equation*}

Consider the inverse operator $L_{\omega,\alpha}^{-1}$. Formally we have
\begin{equation*}
L_{\omega,\alpha}^{-1} f(t,x) = \sum_{(n,k) \in A} \frac{f_{n,k}}{k^{2} + m - \omega^{2} n^{2} + i \omega \alpha nk^2} e^{int} \sin(kx).
\end{equation*}
It is natural that we take the domain of $L_{\omega,\alpha}^{-1}$ to be
\begin{equation*}
\mathscr{D}^{s} ( L_{\omega,\alpha}^{-1} ) = \{ f \in X^{s} \ \vert \ \supp \widehat{f} \subseteq A \} = V^{s},
\end{equation*}
where $\supp \widehat{f}$ denotes the support of the Fourier transform of $f$.

We set
\begin{equation*}
\vartheta(n,k,m, \omega,\alpha) := k^{2} + m - \omega^{2} n^{2} +i \omega \alpha nk^2,
\end{equation*}
so that
\begin{equation}\label{invexp}
L_{\omega,\alpha}^{-1} f(t,x) = \sum_{(n,k) \in A} \frac{f_{n,k}}{\vartheta(n,k,m,\omega,\alpha)} e^{int} \sin(kx).
\end{equation}
Note that
\begin{equation}
\label{lem32pfeq1}
\vert \vartheta(n,k,m, \omega, \alpha) \vert^{2} = (k^{2} - \omega^{2} n^{2} + m)^{2} + \omega^2 \alpha^2 n^2k^4.
\end{equation}

In Section \ref{SectionLSReduction} below we will write our solution $u(t,x)$ of \eqref{omegadnlw} as
\begin{equation*}
    u(t,x) = \frac{\rho}{2} ( e^{it} + e^{-it}  ) \sin(x) + v(t,x),
\end{equation*}
where $\rho \in \R$ and $v \in V^{s}$. The details of this decomposition will be discussed further below.

It is at this stage that we make the following key assumptions regarding the values of the parameters. 

\begin{assumption}
\label{Assumption1}
Recall that we assume $m > 0$ is irrational. There exist $0 < W_{0} < W_{1}$ such that $W_{0} < \omega_{1} = \sqrt{1 + m} < W_{1}$ and $1 < W_{0}^{2} \leq \omega^{2} \leq W_{1}^{2} < m+2$.
\end{assumption}

\begin{assumption}
\label{Assumption2}
The parameter $\rho \in (0,1)$ satisfies $0 < W_{0}^{2p} \theta \rho^{2p}/4 \leq \alpha$, where $\theta = \theta(p) > 0$ is a constant (depending only on $p$) that will be determined below.
\end{assumption}

\begin{assumption}
\label{Assumption3}
Given $k_0\ge 2$ from Theorem \ref{aux_maintheorem}, we assume that the Sobolev index $$s\in [k_0+10,k_0+20].$$
\end{assumption}

We conclude this section with the following result concerning the inverse operator $L_{\omega,\alpha}^{-1}$ which we will require in the next section. In what follows we will often use the notation $\Pi_{k > K}$ to denote the projection onto Fourier modes for which $k > K$.

\begin{lemma} \label{invlemma}
Let $s \geq 2$. 
\begin{equation}
\label{invlemmaeq1}
L_{\omega, \alpha}: V^{s} \to V^{s-3},
\end{equation}
while
\begin{equation}
\label{invlemmaeq2}
L_{\omega, \alpha}^{-1}: V^{s} \to V^{s+1}.
\end{equation}
Let $m > 0$ be irrational. Then under Assumptions \ref{Assumption1} and \ref{Assumption2} above, if $\rho\in(0,1)$ is sufficiently small (depending only on $m,p$) we have the bound

\begin{equation}
\label{invlemmaeq3}
\| L_{\omega,\alpha}^{-1} \|_{\mathcal{L}(V^{s}, V^{s+1})} \lesssim \frac{1}{\rho^{2p}}.
\end{equation}where the implicit constant only depends on $m,p$.

Furthermore, let $K$ be any positive integer. Under the same assumptions we also have the bound
\begin{equation}
\label{invlemmaeq4}
\| L_{\omega,\alpha}^{-1}\Pi_{k>K} \|_{\mathcal{L}(\Pi_{k>K}V^{s}, V^{s+1})} \lesssim \frac{1}{K^2\rho^{2p}}+2.
\end{equation}where the implicit constant still only depends on $m,p$.
\end{lemma}

\begin{proof}
We will show the last bound since the second one will follow similarly. 
Note that if $v\in V^s$ then we have by the expansion \eqref{invexp}
$$\| L_{\omega,\alpha}^{-1}v \|^2_{V^{s+1}} = \sum_{(n,k > K)\in A}  \frac{(n^{2(s+1)}+k^{2(s+1)})\vert v_{n,k} \vert^2}{\vert \vartheta(n,k,m,\omega,\alpha) \vert ^2}$$
and hence it suffices to show that the quantities
\begin{equation*}
\frac{n^2}{\vert \vartheta(n,k,m,\omega,\alpha) \vert ^2} \quad \text{and} \quad \frac{k^2}{\vert \vartheta(n,k,m,\omega,\alpha) \vert ^2}.
\end{equation*}
are bounded by an appropriate constant depending only on $\rho,p,$ and $m$.

Recall that 
$$\vert \vartheta(n,k,m, \omega, \alpha) \vert^{2} = (k^{2} - \omega^{2} n^{2} + m)^{2} + \omega^2 \alpha^2 k^4n^2.$$

We first establish the desired bounds on
\begin{equation*}
   \frac{n^2}{\vert \vartheta(n,k,m,\omega,\alpha) \vert ^2} 
\end{equation*}
If $n=0$ there is nothing to prove. Otherwise $\vert n \vert >0$ and we first note that 
\begin{equation}\label{n_bound}
\frac{n^2}{\vert \vartheta(n,k,m,\omega,\alpha) \vert ^2}\le \frac{n^2}{\omega^2 \alpha^2 k^4n^2}=\frac{1}{\omega^2k^4\alpha^2}< \frac{1}{K^4\alpha^2} \lesssim _{m,p} \frac{1}{K^4\rho^{4p}}
\end{equation}
where we used Assumptions \ref{Assumption1} and \ref{Assumption2}.

It remains to establish a bound on 
\begin{equation*}
    \frac{k^2}{\vert \vartheta(n,k,m,\omega,\alpha) \vert ^2}.
\end{equation*}
First suppose that 
$$\frac{k^4}{(k^{2} - \omega^{2} n^{2} + m)^{2}}\le 4.$$
As $k^2\le k^4$ we are done, since
\begin{equation*}
    \frac{k^{2}}{\vert \vartheta(n,k,m,\omega,\alpha) \vert ^2} \leq \frac{k^{2}}{(k^{2} - \omega^{2} n^{2} + m)^{2}} \leq \frac{k^4}{(k^{2} - \omega^{2} n^{2} + m)^{2}}.
\end{equation*}
On the other hand if 
$$\frac{k^4}{(k^{2} - \omega^{2} n^{2} + m)^{2}} > 4$$ then 
$$k^2 > 2 \vert k^{2}+m - \omega^{2} n^{2} \vert\ge 2k^2+2m-2W_1^2n^2$$ 
and hence
$$k^2\le 2W_1^2n^2-2m\le 2(2+m)n^2-2m < 2(2+m)n^2,$$
where we used that $W_1^2<m+2$ from Assumption \ref{Assumption1}.

It follows that in this case 
$$\frac{k^2}{\vert \vartheta(n,k,m,\omega,\alpha) \vert ^2}< \frac{2(2+m)n^2}{\vert \vartheta(n,k,m,\omega,\alpha) \vert ^2}\lesssim_{m,p} \frac{1}{K^4\rho^{4p}},$$
where the last part follows from \eqref{n_bound}. Hence 
\begin{align}\label{max_of_two}
    \frac{n^2}{\vert \vartheta(n,k,m,\omega,\alpha) \vert ^2} \quad, \quad \frac{k^2}{\vert \vartheta(n,k,m,\omega,\alpha) \vert ^2} \le \max\left \{4,C(m,p)\frac{1}{K^4\rho^{4p}} \right \}
\end{align}
for some constant $C(m,p)>0$ and \eqref{invlemmaeq4} follows. We note that to obtain \eqref{invlemmaeq3} we can repeat the same argument but omit the addition of 2 by taking $\rho\in (0,1)$ sufficiently small depending only on $m,p$ so that the maximum of the two quantities in \eqref{max_of_two} is not 4. The reason we did not do this above for \eqref{invlemmaeq4} is to avoid the dependence of $\rho$ on $K$. 
\end{proof}

\section{The Lyapunov-Schmidt Reduction}
\label{SectionLSReduction}

As in Section \ref{SectionAnsatz} we write $X^{s} = \ker L_{\omega_{1},0} \oplus V^{s}$, the orthogonal sum being with respect to the $L^{2}(\T^{2})$ inner product. Denote by
\begin{equation*}
    \Pi_{\omega_{1},0} : X^{s} \to \ker L_{\omega_{1},0} \qquad \text{and} \qquad \Pi_{V^{s}}: X^{s} \to V^{s},
\end{equation*}
the relevant projection operators. If $u(t,x) \in X^{s}$ solves \eqref{omegadnlw}, then we write
\begin{equation*}
    u(t,x) = \frac{\rho}{2} (e^{it} + e^{-it}) \sin(x) + v(t,x),
\end{equation*}
where $v \in V^{s}$ and $\rho > 0$ is a real amplitude parameter. As the projection operators commute with $L_{\omega, \alpha}$ we obtain the system
\begin{align}
    L_{\omega,\alpha} \rho \cos(t) \sin(x) &= \Pi_{\omega_{1},0} \omega^{2p+1} \Big( \partial_{t} [ \rho \cos(t) \sin(x) + v(t,x) ]\Big )^{2p+1} \label{bifeq}\\
    L_{\omega,\alpha} v(t,x) &= \Pi_{V^{s-1}} \omega^{2p+1} \Big( \partial_{t} [ \rho \cos(t) \sin(x) + v(t,x) ]\Big )^{2p+1}. \label{rangeeq_g}
\end{align}
Equation \eqref{rangeeq_g} is referred to as the range equation while equation \eqref{bifeq} is called the bifurcation equation.

\subsection{Solution of the Range Equation}
The Implicit Function Theorem is not applicable as in the usual Hopf Bifurcation Theorem because when the damping $\alpha=0$ there is a genuine small divisor problem since the numbers $k^2+m-\omega^2n^2$ accumulate to zero. The novelty of this work is to show that this singularity can be removed owing to a certain integral vanishing (see Lemma \ref{large_k_zeros}), in addition to having a suitable lower bound for $\alpha$ in terms of $\rho$ (see Assumption \ref{Assumption2}), and also exploiting the regularizing effect of the Laplacian damping. The removal of this singularity invokes neither Diophantine conditions on the temporal frequency $\omega$, nor any Nash-Moser iteration. We can accomplish this by solving the range equation using the contraction mapping principle. This is a surprising fact since the solutions we will find have damping parameter which accumulates to $\alpha=0$.

Note that \eqref{rangeeq_g} can be written as 
\begin{equation}\label{inv_rangeeq_g}
  v = L_{\omega,\alpha}^{-1}\Pi_{V^{s-1}} \omega^{2p+1} \Big ( \partial_{t} [\rho \cos(t) \sin(x) + v] \Big )^{2p+1}.
\end{equation}

We begin with some Lemmas we will need to solve the range equation \eqref{inv_rangeeq_g}.

\begin{lemma}
\label{large_k_zeros}
There is an integer $K_{p} \geq 1$ such that for each $k > K_{p}$ we have
\begin{equation*}
    \int_{\T} \sin^{2p+1}(x) \sin(kx) dx = 0.
\end{equation*}
\end{lemma}

\begin{proof}
Using the formula 
$$ \sin^{2p+1}(x) \sin(kx) = i 2^{-2 p - 2} (e^{-i k x} - e^{i k x}) i^{2p+1} (e^{-i x} - e^{i x})^{2 p + 1}$$
as well as the Binomial Theorem to expand the term $(e^{-i x} - e^{i x})^{2 p + 1}$ as a sum of exponentials multiplied by constants, it is not difficult to see that we obtain the desired result as long as $k>2p+1$ since the exponents in the exponentials do not vanish, and using periodic boundary conditions. 
\end{proof}

In what follows we will often use the notation $\Pi_{k > K_{p}}$ to denote the projection onto Fourier modes for which $k > K_{p}$.

\begin{lemma}\label{vanish}
We have that \begin{equation*}
    \Pi_{k > K_{p}} \Pi_{V^{s}} \sin^{2p+1}(t) \sin^{2p+1}(x) = 0.
\end{equation*}
\end{lemma}
\begin{proof}
For the first part, as a consequence of Lemma \ref{large_k_zeros} we find that if $k > K_{p}$, then for any $n \in \Z$
\begin{align*}
    &\int_{\T^{2}} \sin^{2p+1}(t) \sin^{2p+1}(x) e^{int} \sin(kx) dt dx \\
    = & \left ( \int_{\T} e^{int} \sin^{2p+1}(t) dt \right ) \left ( \int_{\T} \sin^{2p+1}(x) \sin(kx) dx \right ) = 0,
\end{align*}
where we have used Fubini's Theorem to factor the integrals. It follows that
\begin{equation*}
    \Pi_{k > K_{p}} \Pi_{V^{s}} \sin^{2p+1}(t) \sin^{2p+1}(x) = 0.
\end{equation*}
\end{proof}

For $(n,k) \in A$, we define
\begin{equation}\label{rational_exp}
    Q(n,k) = \frac{n^{2}}{(k^{2} + m - \omega^{2} n^{2})^{2}} \quad \text{and} \quad R(n,k) = \frac{k^{2}}{(k^{2} + m - \omega^{2} n^{2})^{2}}.
\end{equation}

\begin{lemma}
\label{QRbound_small_k}
Let $K > K_p$ be a positive integer and let $(n,k) \in A$ with $k \leq K$. Then provided $0<W_1^2-W_0^2$ is small enough (depending on $K,m,p$) there is a constant $C(K, m, p)$ which is independent of $\rho$ such that
\begin{equation*}
    Q(n,k), R(n,k) \leq C(K, m, p).
\end{equation*}
\end{lemma}

\begin{proof}
We begin by focusing on the bound for $Q(n,k)$.
If $n=0$ there is nothing to prove. Otherwise if $n\neq 0$ and 
\begin{equation*}
    n^2\ge \left ( \frac{K^2+1+m}{W_0^2} \right ) \vert n \vert
\end{equation*}
then 
\begin{equation*}
     \frac{\omega^2n^2-(k^2+m)}{\vert n \vert}\ge \frac{W_0^2n^2-(K^2+m)}{\vert n \vert}\ge (K^2+1+m)-\frac{(K^2+m)}{\vert n \vert}\ge 1
\end{equation*}
 Hence 
\begin{equation*}
    \frac{n^2}{(\omega^2n^2-(k^2+m))^2}\le 1.
\end{equation*}
On the other hand if $n\neq 0$ and 
\begin{equation*}
    n^2< \left ( \frac{K^2+1+m}{W_0^2} \right )\vert n \vert
\end{equation*}
hence by Assumption \ref{Assumption1} since $W_0>1$,
$$\vert n \vert < \left ( \frac{K^2+1+m}{W_0^2} \right ) < K^2+1+m.$$

 It follows both $\vert n\vert,k$ are bounded above by constants which only depend on $K,m,p$; that is $1\le \vert n \vert, k\le A(K,m,p)$ for some $ A(K,m,p)>0$. In this case we can choose $W^2_1-W^2_0>0$ small enough so that 
 \begin{equation*}
      \vert \omega^2-(1+m)\vert < \min_{\vert n \vert ,k \le  A(K,m,p)}\frac{1}{2}\left | \frac{k^2+m}{n^2}-(1+m)\right |
 \end{equation*}
Then it is easily seen that 
$$\vert \omega^2n^2-(k^2+m) \vert\ge \min_{\vert n \vert ,k \le  A(K,m,p)}\frac{1}{2}\left | \frac{k^2+m}{n^2}-(1+m)\right | >0,$$ 
where the last part is positive since $m$ is assumed irrational. Hence the quantities $Q(n,k)$ are bounded, above by a constant depending only on $m,p$.
Lastly, the statement concerning the quantities $R(n,k)$ follows because $k^2$ is bounded and  we just showed the denominators are bounded away from zero.
\end{proof}

\begin{lemma}\label{small_inv_lemma}
Let $K>K_p$ be given. Then provided $0<W_1^2-W_0^2$ is small enough (depending on $K,m,p$) there is some $C(K,m,p) > 0$ independent of $\rho$ such that
\begin{align*}
\| L_{\omega,\alpha}^{-1} \Pi_{k \le K} \|_{\mathcal{L}(\Pi_{k \le K}V^{s-1}, V^s)} & \le C(K,m,p).
\end{align*}
That is, the inverse restricted to the said region is bounded independently of $\rho$.
\end{lemma}

\begin{proof}
This follows from Lemma \ref{QRbound_small_k} (recall also the beginning of the proof of Lemma \ref{invlemma}).
\end{proof}

Our strategy will be to decompose the range equation according to the spatial frequency $k$, thereby enabling us to obtain a contraction. Let 
\begin{equation}\label{nonlinear_map}
    \mathcal{A}(v) = L_{\omega,\alpha}^{-1} \Pi_{V^{s-1}} \omega^{2p+1} \Big ( \partial_{t} [\rho \cos(t) \sin(x) + v(t,x) ] \Big )^{2p+1}.
\end{equation}

We begin with a proposition which establishes a bound for the above nonlinear map at zero. 

\begin{proposition}\label{nonlinear_zero}
If $0<W_1^2-W_0^2$ is small enough (depending on $m,p$) then there is some constant $C(p,m,k_0)>0$ depending only on $p,m,k_0$ so that 

$$\| \mathcal{A}(0) \|_{V^s} \le C(p,m,k_0) \rho^{2p+1}.$$
\end{proposition}

\begin{proof}
We let $K_{p}$ be the parameter from Lemma \ref{vanish} and $K = K_{p} + 1$. We decompose $\mathcal{A}$ as follows:
\begin{equation}\label{A_decompose} 
    \mathcal{A}(v) = \mathcal{A}_{\leq K}(v) + \mathcal{A}_{> K}(v),
\end{equation}
where
\begin{align*}
   \mathcal{A}_{\leq K}(v) &= L_{\omega,\alpha}^{-1} \Pi_{k \leq K} \Pi_{V^{s-1}} \omega^{2p+1} \Big ( \partial_{t} [\rho \cos(t) \sin(x) + v(t,x) ] \Big )^{2p+1},\\
   \mathcal{A}_{>K}(v) &= L_{\omega,\alpha}^{-1} \Pi_{k > K} \Pi_{V^{s-1}} \omega^{2p+1} \Big ( \partial_{t} [\rho \cos(t) \sin(x) + v(t,x) ] \Big )^{2p+1}.
\end{align*}

In particular 

\begin{align*}
 \| \mathcal{A}(0) \|_{V^s} & \le \| \mathcal{A}_{\leq K}(0) \|_{V^s} + \| \mathcal{A}_{> K}(0) \|_{V^s}\\
& = \| L_{\omega,\alpha}^{-1} \Pi_{k \leq K}\Pi_{V^{s-1}} \omega^{2p+1} \Big ( -\rho \sin(t) \sin(x) \Big )^{2p+1} \|_{V^s}  \\
& + \| L_{\omega,\alpha}^{-1} \Pi_{k > K}\Pi_{V^{s-1}} \omega^{2p+1} \Big ( -\rho \sin(t) \sin(x) \Big )^{2p+1} \|_{V^s} \\
& = \| L_{\omega,\alpha}^{-1} \Pi_{k \leq K}\Pi_{V^{s-1}} \omega^{2p+1} \Big ( -\rho \sin(t) \sin(x) \Big )^{2p+1} \|_{V^s} \\
& \le C(p,m,k_0)\rho^{2p+1}
\end{align*}
having used Lemmas \ref{vanish}, \ref{small_inv_lemma}, the Banach algebra property \ref{algebra}, and Assumptions \ref{Assumption1} and \ref{Assumption3}.

\end{proof}

\begin{proposition}\label{contraction}
Let $\rho>0$ and $0<W_1^2-W_0^2$ be sufficiently small (depending on $p,m,k_0$). Then 
\begin{equation*}
    \mathcal{A}:V^{s} \to V^{s} \qquad \text{and} \qquad \mathcal{A}:\mathscr{B}_{V^{s}}(0,\rho^{2p+\frac{1}{2}}) \to \mathscr{B}_{V^{s}}(0,\rho^{2p+\frac{1}{2}}).
\end{equation*}
Moreover, if $\rho$ is sufficiently small (depending on $p,m,k_0,$) and $v,w \in \mathscr{B}_{V^{s}}(0,\rho^{2p+\frac{1}{2}})$, then 
\begin{equation}
\label{Acontracts}
    \| \mathcal{A}(v) - \mathcal{A}(w) \|_{V^{s}} \le \frac{1}{2} \| v - w \|_{V^{s}}.
\end{equation}
\end{proposition}

\begin{proof}
From Lemma \ref{invlemma} it follows that $\mathcal{A}:V^{s} \to V^{s}$.  We first prove \eqref{Acontracts}. Note that from \ref{nonlinear_map} we have 
\begin{equation*}
    D \mathcal{A}(v) \zeta =  L_{\omega, \alpha}^{-1} \Pi_{V^{s-1}} \left [\omega^{2p+1} (2p + 1) \Big ( \partial_{t} [ \rho \cos(t) \sin(x) + v] \Big )^{2p} \partial_{t}\zeta \right ]
\end{equation*}
It suffices to show that for $v\in \mathscr{B}_{V^{s}}(0,\rho^{2p+\frac{1}{2}})$, 
$\|D \mathcal{A}(v)\|_{\mathcal{L}(V^{s}, V^{s})}\le \frac{1}{2}$, whence the contraction estimate holds thanks to the Mean Value Inequality. Let $K$ be a positive integer which will be chosen later.
\begin{align*}
    & D \mathcal{A}(v) \zeta \\
    & =  L_{\omega, \alpha}^{-1} \Pi_{V^{s-1}} \left [\omega^{2p+1} (2p + 1) \Big ( \partial_{t} [ \rho \cos(t) \sin(x) + v] \Big )^{2p} \partial_{t}\zeta \right ]\\
    & = L_{\omega, \alpha}^{-1} \Pi_{k\le K} \Pi_{V^{s-1}}  \left [\omega^{2p+1} (2p + 1) \Big ( \partial_{t} [ \rho \cos(t) \sin(x) + v] \Big )^{2p} \partial_{t}\zeta \right ]\\
    & + L_{\omega, \alpha}^{-1} \Pi_{k> K} \Pi_{V^{s-1}} \left [\omega^{2p+1} (2p + 1) \Big ( \partial_{t} [ \rho \cos(t) \sin(x) + v] \Big )^{2p} \partial_{t}\zeta \right ]\\
    & = (D \mathcal{A}(v) \zeta)_{k\le K} + (D \mathcal{A}(v) \zeta)_{k> K},
\end{align*}

where the terms of the last line are defined in the corresponding way by the terms from the line preceding it.

If $v\in \mathscr{B}_{V^{s}}(0,\rho^{2p+\frac{1}{2}})$ then for any $\zeta\in V^{s}$ using \eqref{algebra}, Lemma \ref{invlemma}, Assumptions \ref{Assumption1} and  \ref{Assumption3}, as well as $\|\partial_{t}\zeta\|_{V^{s-1}}\lesssim \|\zeta\|_{V^{s}}$, we have 
\begin{align*}
     &\| (D \mathcal{A}(v)\zeta)_{k > K} \|_{V^{s}} \le \\ 
     & C(p,m,k_0) (\frac{1}{K^2\rho^{2p}} +2) 
     \left \| \Pi_{V^{s-1}}\left [ \Big ( -\rho \sin(t) \sin(x) + \partial_{t} v \Big )^{2p} \partial_{t}\zeta \right ] \right \|_{V^{s-1}}\\
     & \le C(p,m,k_0) \frac{1}{K^2\rho^{2p}}\left \| ( -\rho \sin(t) \sin(x) + \partial_{t} v \Big )^{2p} \right \|_{V^{s-1}} \left \| \partial_{t}\zeta \right \|_{V^{s-1}}\\
     & + 2C(p,m,k_0)\left \| ( -\rho \sin(t) \sin(x) + \partial_{t} v \Big )^{2p} \right \|_{V^{s-1}} \left \| \partial_{t}\zeta \right \|_{V^{s-1}}\\
     & \le C(p,m,k_0) \frac{\rho^{2p}}{K^2\rho^{2p}}\left \|\zeta \right \|_{V^{s}} + 2C(p,m,k_0)\rho^{2p}\left \|\zeta \right \|_{V^{s}}\\ = & C(p,m,k_0) \frac{1}{K^2}\left \|\zeta \right \|_{V^{s}} + 2C(p,m,k_0)\rho^{2p}\left \|\zeta \right \|_{V^{s}}.
\end{align*}     
\begin{remark}
The constants $C$ above may change from line to line but only depend on the parameters in the parentheses.
\end{remark} 

Hence we can take $K$ to be large enough (depending only on $p,m,k_0$) so that $C(p,m,k_0)\frac{1}{K^2}\le \frac{1}{8}$ as well as $\rho\in (0,1)$ small enough (depending on $p,m,k_0$ to ensure $2C(p,m,k_0)\rho^{2p}\le \frac{1}{8}$.

Having selected $K=K(p,m,k_0)$, next we have by Lemma \ref{small_inv_lemma}, using $v\in \mathscr{B}_{V^{s}}(0,\rho^{2p+\frac{1}{2}})$, \eqref{algebra}, Assumptions \ref{Assumption1} and \ref{Assumption3}, and $\|\partial_{t}\zeta\|_{V^{s-1}}\lesssim \|\zeta\|_{V^{s}}$ ,
\begin{align*}
     &\| (D \mathcal{A}(v)\zeta)_{k \le K} \|_{V^{s}} \le 
     C(p,m,k_0)\left \| \Pi_{V^{s-1}}\left [ \Big ( -\rho \sin(t) \sin(x) + \partial_{t} v \Big )^{2p} \partial_{t}\zeta \right ] \right \|_{V^{s-1}}\\
     & \le C(p,m,k_0)\rho^{2p}\left \|\zeta \right \|_{V^{s}}.
\end{align*}     

We choose $\rho\in(0,1)$ (depending only on $p,m,k_0$) small enough so that $C(p,m,k_0)\rho^{2p}\le \frac{1}{4}$. It follows that 
$$\sup_{v\in \mathscr{B}_{V^{s}}(0,\rho^{2p+\frac{1}{2}})}\| D \mathcal{A}(v)\|_{\mathcal{L}(V^{s},V^{s})}\le\frac{1}{2}.$$ 

Finally, to complete the first part of the proof suppose that $v\in \mathscr{B}_{V^{s}}(0,\rho^{2p + \frac{1}{2}})$ . We
first observe from Proposition \ref{nonlinear_zero} that 
$$\| \mathcal{A}(0) \|_{V^{s}} \le C(m,p,k_0)\rho^{2p+1}$$
and hence using the above we have
\begin{align*}
    & \| \mathcal{A}(v) \|_{V^{s}} \le \| \mathcal{A}(v) - \mathcal{A}(0) \|_{V^{s}} + \| \mathcal{A}(0) \|_{V^{s}} \\
    & \le 
     \frac{1}{2}\| v\|_{V^{s}} + \| \mathcal{A}(0) \|_{V^{s}} \le \frac{1}{2}\rho^{2p+\frac{1}{2}} + C(m,p,k_0)\rho^{2p+1} \le \rho^{2p+\frac{1}{2}}
\end{align*}
provided $$C(m,p,k_0)\rho^{\frac{1}{2}} \le \frac{1}{2}$$
which can be accomplished if $\rho\in (0,1)$ is chosen small enough, depending only on $m,p,k_0$.

\end{proof}

Altogether we are lead to the following result.

\begin{proposition}
\label{good_re_sol}
Under Assumptions \ref{Assumption1}, \ref{Assumption2}, and \ref{Assumption3} if $\rho>0$ and $0<W_1^2-W_0^2$ are sufficiently small (depending on $p,m,k_0$)  the range equation \eqref{rangeeq_g} (as well as \eqref{inv_rangeeq_g}) has a solution $v(\rho, \omega, \alpha)$ satisfying
\begin{equation}
\label{v_estimate}
    \| v \|_{V^{s}} \le \rho^{2p+\frac{1}{2}}.
\end{equation}

\end{proposition}

 Using Proposition \ref{contraction} we see that the map $\mathcal{A}$ is a contraction on $\mathscr{B}_{V^{s}}(0,\rho^{2p+\frac{1}{2}})$, and hence has a fixed point inside this ball which is a solution of the range equations \eqref{rangeeq_g}, \eqref{inv_rangeeq_g} which clearly depends on $(\rho, \omega, \alpha)$.

 Next we derive estimates on the derivatives of the nonlinear map defined in \eqref{nonlinear_map} as well as the solution of the range equation granted above, which will be handy for solving the bifurcation equations downstream. 

\begin{lemma}\label{est_der_nonlinear_map}
 Let $v$ be the solution of the range equation granted in Proposition \ref{good_re_sol}. Omitting the dependence on $\rho, \omega$ for the sake of brevity we have that under Assumptions \ref{Assumption1}, \ref{Assumption2}, and \ref{Assumption3} if $\rho>0$ and $0<W_1^2-W_0^2$ are sufficiently small (depending on $p,m,k_0$) then
\begin{equation*}    
    \Big \| \partial_{\alpha}\mathcal{A}(\alpha,v)\Big \vert _{v=v(\alpha)} \Big \|_{V^s}\le C(p,m,k_0)\rho^{\frac{1}{2}} \end{equation*}
for some $C(p,m,k_0)>0$  where $\mathcal{A}$ is given in \eqref{nonlinear_map}, but we call out the dependence on $\alpha$ explicitly here.
\end{lemma}

\begin{proof}
    We see that for general $v\in V^s$
    \begin{align}\label{partial_A_alpha}
  \partial_{\alpha}\mathcal{A}(\alpha,v) = (\partial_{\alpha}L_{\omega,\alpha}^{-1}) \Pi_{V^{s-1}} \omega^{2p+1} \Big ( \partial_{t} [\rho \cos(t) \sin(x) + v(t,x) ] \Big )^{2p+1}       
  \end{align}
  and it is not difficult to see that for any $v\in V^s$

  \begin{align*}
      \partial_{\alpha}L_{\omega,\alpha}^{-1}v = \sum_{n,k} \frac{-i \omega n k^2}{(k^2+m-\omega^2n^2+i\omega n \alpha k^2)^2}v_{nk}.
  \end{align*}
   See the appendix of \cite{kospigott_highd_sym} for details. We establish the following properties:  
   \begin{equation*}
       \mbox{if} \ v\in V^s \ \mbox{then} \   \partial_{\alpha}L_{\omega,\alpha}^{-1}v \in V^{s+1},\; \mbox{and} \; \| \partial_{\alpha}L_{\omega,\alpha}^{-1} \|_{\mathcal{L}(V^s,V^{s+1})} \le C(p,m,k_0) \frac{1}{W_0 \rho^{4p}}
   \end{equation*}

   To see this it suffices to show that the quantities
   \begin{align}\label{alpha_rational_exp}
& n^2 \cdot \frac{\omega^2 n^2 k^4}{((k^2+m-\omega^2n^2)^2+(\omega n \alpha k^2)^2)^2},
  k^2 \cdot \frac{\omega^2 n^2 k^4}{((k^2+m-\omega^2n^2)^2+(\omega n \alpha k^2)^2)^2}
  \end{align}
  are bounded by an appropriate constant. It is easily checked that for both we have the upper bound $\frac{1}{\omega^2\alpha^4}$ from which the second property above follows from Assumption \ref{Assumption2} (and taking the square root).

  Finally,  by \eqref{partial_A_alpha} 
  \begin{align*}
 &\partial_{\alpha}\mathcal{A}(\alpha,v)\\ 
 = &(\partial_{\alpha}L_{\omega,\alpha}^{-1}) \Pi_{V^{s-1}} \omega^{2p+1} \Big ( (-\rho \sin(t)\sin(x))^{2p+1} \\
 &\qquad \qquad \qquad \quad \qquad + \sum_{j=1}^{2p+1} \binom{2p+1}{j}(-\rho \sin(t)\sin(x))^{2p+1-j}(\partial_t v)^{j} \Big ) \\
& = (\partial_{\alpha}L_{\omega,\alpha}^{-1}) \Pi_{V^{s-1}} \omega^{2p+1} (-\rho \sin(t)\sin(x))^{2p+1}  \\
& + (\partial_{\alpha}L_{\omega,\alpha}^{-1}) \Pi_{V^{s-1}} \omega^{2p+1} \Big (\sum_{j=1}^{2p+1} \binom{2p+1}{j}(-\rho \sin(t)\sin(x))^{2p+1-j}(\partial_t v)^{j} \Big ) \\
& =: (I) + (II)
  \end{align*}
   For the first term $(I)$ we see that by Lemma \ref{vanish} we see that we need to only consider at most $K_p+1$ modes in $k$, that is, 
   \begin{equation*}
   (I)= (\partial_{\alpha}L_{\omega,\alpha}^{-1}) \Pi_{k\le K_p+1}\Pi_{V^{s-1}} \omega^{2p+1} (-\rho \sin(t)\sin(x))^{2p+1}. 
   \end{equation*}
    To obtain a bound for this we observe that the quantities in \ref{alpha_rational_exp} are in fact bounded by $D(p,m)Q(n,k)^2, D(p,m)R(n,k)^2$ respectively for some constant $D(p,m)>0$, where $Q(n,k), R(n,k)$ are the rational expressions defined in \eqref{rational_exp} and the latter are, by Lemma \ref{QRbound_small_k}, bounded by a constant depending only on $p,m$. This means that $\| (I) \|_{V^{s}}\le D(p,m)\rho^{2p+1}$ for some possibly different $D(p,m)>0$.

   Turning to $(II)$ and recalling we are interested in the case $v=v(\alpha)$ we have that since $\| v(\alpha) \|_{V^s} \le \rho^{2p+\frac{1}{2}}$
   \begin{align*}
       & \Big \| (\partial_{\alpha}L_{\omega,\alpha}^{-1}) \Pi_{V^{s-1}} \omega^{2p+1} \Big (\sum_{j=1}^{2p+1} \binom{2p+1}{j}(-\rho \sin(t)\sin(x))^{2p+1-j}(\partial_t v)^{j} \Big )\Big \|_{V^s} \\
       & \le E(p,m,k_0)\frac{1}{\rho^{4p}}\sum_{j=1}^{2p+1}\rho^{2p+1-j}(\rho^{2p+\frac{1}{2}})^j\le E(p,m,k_0)\rho^{\frac{1}{2}}
   \end{align*}
   for some $E(p,m,k_0)>0$. Hence
   $$\Big \| \partial_{\alpha}\mathcal{A}(\alpha,v)\Big \vert _{v=v(\alpha)} \Big \|_{V^s}\le D(p,m)\rho^{2p+1}+ E(p,m,k_0)\rho^\frac{1}{2}\le 2E(p,m,k_0)\rho^{\frac{1}{2}}$$

   giving the desired result if $\rho\in(0,1)$ is small enough depending only on $p,m,k_0$.
\end{proof}

 \begin{proposition}\label{est_der_rangeeq}
     Omitting the dependence on $\rho, \omega$ for the sake of brevity we have that under Assumptions \ref{Assumption1}, \ref{Assumption2}, and \ref{Assumption3} if $\rho>0$ and $0<W_1^2-W_0^2$ are sufficiently small (depending on $p,m,k_0$) then $\| \partial_{\alpha} v(\alpha) \|_{V^s}\le C(p,m,k_0) \rho^\frac{1}{2}$
 \end{proposition}

 \begin{proof}
 We have 
\begin{equation*}
    v(\alpha) - \mathcal{A}(\alpha,v(\alpha)) = 0,
\end{equation*}
where $\mathcal{A}$ is given in \eqref{nonlinear_map} but we call out the dependence on $\alpha$ explicitly here.  
We remark that the invertibility of
\begin{equation*}
    \partial_{v} \Big ( v - \mathcal{A}(\alpha,v) \Big ) \Big \vert_{v =v(\alpha)}= \Big (I_{V^s}-\partial_{v}\mathcal{A}(v)\Big )\Big \vert _{v=v(\alpha)}
\end{equation*}
and the fact that 

\begin{equation*}
    \Big \| \Big (I_{V^s}-\partial_{v}\mathcal{A}(v)\Big )^{-1}\Big \vert _{v=v(\alpha)}\Big \|_{\mathcal{L}(V^s,V^s)}\le 2
\end{equation*}
follows from a Neumann series argument using the estimate \eqref{v_estimate} on $v$, the same argument as in the first part of the proof of Proposition \ref{contraction} (where we showed there that $\left \| \partial_{v}\mathcal{A}(v)\Big \vert _{v=v(\alpha)} \right \|_{\mathcal{L}(V^s,V^s)}\le \frac{1}{2})$,  and taking $\rho>0$ to be sufficiently small.  

Hence it follows that 
\begin{align*}
    & \partial_{\alpha}v(\alpha) = -\Big (I_{V^s}-\partial_{v}\mathcal{A}(v)\Big )^{-1}\Big \vert _{v=v(\alpha)}\partial_{\alpha}\mathcal{A}(\alpha,v)\Big \vert _{v=v(\alpha)} \implies \\
    & \| \partial_{\alpha}v \|_{V^s} \le 2 \Big \| \partial_{\alpha}\mathcal{A}(\alpha,v)\Big \vert _{v=v(\alpha)} \Big \|_{V^s}\le C(p,m,k_0)\rho^{\frac{1}{2}}
\end{align*}

for some $C(p,m,k_0)>0$ by Lemma \ref{est_der_nonlinear_map}.
\end{proof}

\subsection{Solution of the Bifurcation Equation}
To solve the bifurcation equation \eqref{bifeq} we first reformulate the problem as a pair of equations. The left side of \eqref{bifeq} can be calculated explicitly:
\begin{equation}
    \label{bifeqLHS}
    L_{\omega,\alpha} \rho \cos(t) \sin(x) = (1 - \omega^{2} + m) \rho \cos(t) \sin(x) - \omega \alpha \rho \sin(t) \sin(x).
\end{equation}
One can also rewrite the right side of \eqref{bifeq} using Lemma \ref{kernellemma}:
\begin{equation}
    \label{bifeqRHS}
    \begin{aligned}
    & \omega^{2p+1} \Pi_{\omega_{1},0} \Big ( \partial_{t}[\rho \cos(t) \sin(x) + v] \Big )^{2p+1}\\
    =& A(\rho, \omega, \alpha) \sin(t) \sin(x) + B(\rho, \omega, \alpha) \cos(t) \sin(x),
    \end{aligned}
\end{equation}
where
\begin{equation*}
    A(\rho, \omega, \alpha) = \omega^{2p+1} \int_{\T^{2}} \Big ( \partial_{t}[\rho \cos(t) \sin(x) + v] \Big )^{2p+1} \sin(t) \sin(x) dt dx,
\end{equation*}
and
\begin{equation*}
    B(\rho, \omega, \alpha) = \omega^{2p+1} \int_{\T^{2}} \Big ( \partial_{t}[\rho \cos(t) \sin(x) + v] \Big )^{2p+1} \cos(t) \sin(x) dt dx.
\end{equation*}
Equating coefficients in \eqref{bifeqLHS} and \eqref{bifeqRHS} yields a pair of equations:
\begin{align}
    -\omega \alpha \rho &= \omega^{2p+1} \int_{\T^{2}} \Big ( \partial_{t}[\rho \cos(t) \sin(x) + v] \Big )^{2p+1} \sin(t) \sin(x) dt dx \label{bifeq1}\\
    (m + 1 - \omega^{2}) \rho &= \omega^{2p+1} \int_{\T^{2}} \Big ( \partial_{t}[\rho \cos(t) \sin(x) + v] \Big )^{2p+1} \cos(t) \sin(x) dt dx. \label{bifeq2}
\end{align}
We divide both of these equations by $\rho$; we also divide \eqref{bifeq1} by $\omega$:
\begin{align}
    - \alpha  &= \frac{\omega^{2p}}{\rho} \int_{\T^{2}} \Big ( \partial_{t}[\rho \cos(t) \sin(x) + v] \Big )^{2p+1} \sin(t) \sin(x) dt dx \label{bifeq3}\\
    (m + 1 - \omega^{2}) &= \frac{\omega^{2p+1}}{\rho} \int_{\T^{2}} \Big ( \partial_{t}[\rho \cos(t) \sin(x) + v] \Big )^{2p+1} \cos(t) \sin(x) dt dx. \label{bifeq4}
\end{align}

Before proceeding further we observe that
\begin{equation*}
    \Big ( \partial_{t}[\rho \cos(t) \sin(x) + v] \Big )^{2p+1} = -\rho^{2p+1} \sin^{2p+1}(t) \sin^{2p+1}(x) + F(\rho, t,x),
\end{equation*}
where
\begin{equation}
    \label{Fdefinition}
    F(\rho, t, x) = \sum_{j=1}^{2p+1} \binom{2p+1}{j} \Big ( -\rho \sin(t) \sin(x) \Big )^{2p+1 - j} (v_{t})^{j}.
\end{equation}
Returning to \eqref{bifeq3} and \eqref{bifeq4} we have
\begin{align*}
    -\alpha &= -\omega^{2p} \rho^{2p} \int_{\T^{2}} \sin^{2p+2}(t) \sin^{2p+2}(x) dt dx \\
    & \quad + \frac{\omega^{2p}}{\rho} \int_{\T^{2}} F(\rho,t,x) \sin(t) \sin(x) dt dx,
    \intertext{and}
    m + 1 - \omega^{2} &= -\omega^{2p+1} \rho^{2p} \int_{\T^{2}} \sin^{2p+1}(t) \cos(t) \sin^{2p+2}(x) dt dx \\
    & \quad + \frac{\omega^{2p+1}}{\rho} \int_{\T^{2}} F(\rho,t,x) \cos(t) \sin(x) dt dx.
\end{align*}
Observe that
\begin{equation*}
    \int_{\T^{2}} \sin^{2p+1}(t) \cos(t) \sin^{2p+2}(x) dt dx = 0
\end{equation*}
and let
\begin{equation*}
    \theta = \theta(p) = \int_{\T^{2}} \sin^{2p+2}(t) \sin^{2p+2}(x) dt dx = \left ( \binom{2p+1}{p} \frac{\pi}{2^{2p}} \right )^{2}.
\end{equation*}
It now follows that the bifurcation equations can be written as
\begin{align}
    \alpha &= \omega^{2p} \theta \rho^{2p} - \frac{\omega^{2p}}{\rho} \int_{\T^{2}} F(\rho,t,x) \sin(t) \sin(x) dt dx \label{be1}\\
    m + 1 - \omega^{2} &= \frac{\omega^{2p+1}}{\rho} \int_{\T^{2}} F(\rho,t,x) \cos(t) \sin(x) dt dx. \label{be2}
\end{align}

We take the solution $v(\rho, \omega, \alpha)$ granted by Proposition \ref{good_re_sol} and substitute it into the bifurcation equations \eqref{be1}, \eqref{be2} and note that the dependence on $v$ is through $F$.

\subsubsection{Estimates on $F$}
It will be worthwhile to have the following estimates related to the function $F$ defined in \eqref{Fdefinition}:

\begin{lemma}\label{Fest}
Let $v = v(\rho, \omega, \alpha)$ denote the solution of \eqref{rangeeq_g} obtained in Proposition \ref{good_re_sol}. We have that
\begin{equation*} 
\frac{\vert F(\rho,t,x)\vert }{\rho}\lesssim \rho^{2p+1}
\end{equation*}
where the implicit constant  only depends on $k_0,m,p$.
\end{lemma}

\begin{proof}
Using Proposition \ref{good_re_sol} we have that $\|v\|_{V^{s}} \le \rho^{2p+\frac{1}{2}}\le \rho^2$

hence 
\begin{align*}
\left \| F(\rho, t, x) \right \|_{V^{s-1}} & = \left \| \sum_{j=1}^{2p+1} \binom{2p+1}{j} ( -\rho \sin(t) \sin(x) )^{2p+1-j} (v_{t})^{j} \right \|_{V^{s-1}}\\
& \lesssim \sum_{j=1}^{2p+1} \binom{2p+1}{j}\rho^{2p+1-j}({\rho^2})^j = \sum_{j=1}^{2p+1} \binom{2p+1}{j}\rho^{2p+1+j} \\
& \lesssim \rho^{2p+2}.
\end{align*}
Hence, 
\begin{equation*}
    \frac{\| F(\rho,t,x)\|_{V^{s-1}}}{\rho} \lesssim \rho^{2p+1}
\end{equation*}
and the implicit constant  depends only  on $k_0,p,m$.
\end{proof}

We are now equipped to solve the bifurcation equations \eqref{be1} and \eqref{be2}.

\begin{proposition}
\label{bes_sol}
Let $v = v(\rho, \omega, \alpha)$ denote the solution of \eqref{rangeeq_g} obtained in Proposition \ref{good_re_sol}. Let $\rho > 0$ be sufficiently small; in particular we require that $\rho$ satisfy the inequality \eqref{rho_cond2} below. Then the bifurcation equation \eqref{be1} has a solution $\alpha(\omega, \rho) \gtrsim \rho^{2p} > 0$, where the implicit constant depends only on $k_0,m,p$. Furthermore, for fixed $\rho > 0$ sufficiently small, the map
\begin{equation*}
    [W_{0}, W_{1}] \ni \omega \mapsto \alpha(\omega, \rho) \in (0,\infty)
\end{equation*}
is continuous. The bifurcation equation \eqref{be2} has a solution $\omega(\rho) \in [W_{0}, W_{1}]$. Finally, we have that $\alpha \downarrow 0$ and $\omega^{2} \to \omega_{1}^{2} = m+1$ as $\rho \downarrow 0$.
\end{proposition}

\begin{proof}
Beginning with \eqref{be1} we let
\begin{equation*}
    \mathcal{B}(\alpha) := \omega^{2p} \theta \rho^{2p} - \frac{\omega^{2p}}{\rho} \int_{\T^{2}} F(\rho,t,x) \sin(t) \sin(x) dt dx.
\end{equation*}
Using the fact that $\theta > 0$ along with Lemma \ref{Fest} we have that
\begin{equation*}
    \mathcal{B}(\alpha) \leq \vert \mathcal{B}(\alpha) \vert \le W_{1}^{2p} \theta \rho^{2p} + C_{k_0,m,p}W_{1}^{2p} \rho^{2p+1}.
\end{equation*}
On the other hand,
\begin{align*} 
    \mathcal{B}(\alpha) &\geq W_{0}^{2p} \theta \rho^{2p} - C_{k_0,m,p} W_{1}^{2p} \rho^{2p+1}\\
    &= \rho^{2p}(W_{0}^{2p} \theta - C_{k_0,m,p} W_{1}^{2p} \rho)\\
    & \geq \frac{1}{2} W_{0}^{2p} \theta \rho^{2p},
\end{align*}
where the last inequality holds provided
\begin{equation}
    \label{rho_cond2}
    0 < \rho \leq \frac{1}{2 C_{k_0,m,p}} \left ( \frac{W_{0}}{W_{1}} \right )^{2p} \theta.
\end{equation}
Hence provided $\rho$ satisfies \eqref{rho_cond2}, we see that
\begin{equation*}
    \mathcal{B}(\alpha) \in I = \left [ \frac{1}{2} W_{0}^{2p} \theta \rho^{2p}, W_{1}^{2p} \theta \rho^{2p} + C_{k_0,m,p}W_{1}^{2p} \rho^{2p+1} \right ],
\end{equation*}
for some constant $C_{k_0,m,p} > 0$ depending only on $k_0,m,p$. We conclude that $\mathcal{B}(I) \subset I$. 

\begin{remark}
    Assumption \ref{Assumption2} implies that the domain of $\mathcal{B}$ is $[\frac{1}{4}W_0^{2p}\theta \rho^{2p}, \infty)\supset I$.  Had we changed the exponent in $\rho$ in Assumption \ref{Assumption2}, we would not necessarily be able to conclude that $I$ is contained in the domain of $\mathcal{B}$ and hence would not get a fixed point from these estimates. 
\end{remark}

We will show that actually $\mathcal{B}$ is a contraction on $I$. To see this by Proposition \ref{est_der_rangeeq} we have (using that in particular $\| v(\alpha)\|_{V^s}\le \rho^2$)
\begin{align*}
    & \left \vert \partial_{\alpha}\mathcal{B(\alpha)} \right \vert \lesssim \frac{1}{\rho}\sum_{j=1}^{2p+1} \binom{2p+1}{j} \left \| ( -\rho \sin(t) \sin(x) )^{2p+1-j} j(v(\alpha)_t)^{j-1}\partial_{\alpha}v(\alpha)_{t} \right \|_{V^{s-1}} \\
& \lesssim \frac{1}{\rho}\sum_{j=1}^{2p+1} \binom{2p+1}{j} \rho^{2p+1-j} j(\rho^2)^{j-1}\rho^\frac{1}{2} \lesssim \rho^{2p-\frac{1}{2}},
\end{align*}
so it follows that $\left \vert \partial_{\alpha}\mathcal{B(\alpha)} \right \vert$ can be made smaller than $\frac{1}{2}$ as long as $\rho\in (0,1)$ is small enough depending only on $p,m,k_0$ meaning $\mathcal{B}$ is a contraction. This grants us a unique fixed point for $\mathcal{B}$ on $I$.
Finally we insert the solution $\alpha\in I$ of \eqref{be1} into the bifurcation equation \eqref{be2}. For convenience we will solve the latter for $\omega^2$ so we define 
\begin{equation*} 
\mathcal{G}(\xi):=m + 1-\frac{\xi^\frac{2p+1}{2}}{\rho} \int_{\T^{2}} F(\rho,t,x) \cos(t) \sin(x) dt dx,
\end{equation*}
where $\xi =\omega^2$. As above we note that if $\xi\in [W_0^2,W_1^2]$ then using \ref{Fest}
\begin{equation*}
\mathcal{G}(\xi)\le \vert \mathcal{G}(\xi) \vert \le m+1+D_{k_0,m,p}W_1^{2p+1}\rho^{2p+1}\le W_1^{2}
\end{equation*}

if 
\begin{equation}\label{rho_cond3}
0< \rho^{2p+1} \lesssim \frac{W_1^{2}-(m+1)}{W_1^{2p+1}},
\end{equation}
where we note that the right hand side of the above inequality is positive since $W_1^2>m+1$.
Furthermore, 
\begin{equation*}
\mathcal{G}(\xi)  \ge m+1-D_{k_0,m,p}W_1^{2p+1}\rho^{2p+1}\ge W_0^2
\end{equation*}
if
\begin{equation}\label{rho_cond3}
0< \rho^{2p+1} \le \frac{m+1-W_0^2}{D_{k_0,m,p}W_1^{2p+1}},
\end{equation}
for some constant $D_{k_0,m,p}>0$. Note again that the right hand side above is positive since $m+1>W_0^2$. The continuity of $\mathcal{G}$ on $[W_0^2,W_1^2]$ follows from the fact that $\alpha(\rho,\omega)$ is continuous in $\omega$, which is a consequence of the fact that $\mathcal{B}$ is a contraction (see Corollary 5.2.2 of \cite{hamilton1982nash}), and we conclude the existence a fixed point of $\mathcal{G}$ and hence a solution $\omega(\rho) = \sqrt{\xi}\in [W_0,W_1]$ of \eqref{be2} by the Brouwer fixed point theorem.

The final statement concerning the limits is clear for the bounds we obtained here and their dependence on $\rho$.
\end{proof}

This completes the proof of Theorem \ref{aux_maintheorem} and hence also Theorem \ref{maintheorem}. We note that the statement concerning $C^{k_0}$ regularity of the solutions follows from Sobolev embedding since, in particular, $s >k_0+1$. 

\section{Acknowledgments}

The authors would like to thank the anonymous referee for a careful reading and helpful tips to improve the exposition.

% References
\bibliographystyle{plain}
\bibliography{reflist}

\end{document}